\documentclass[12pt]{amsart}
\usepackage{amssymb}

\theoremstyle{plain}
\newtheorem{thm}[equation]{Theorem}
\newtheorem{lem}[equation]{Lemma}
\newtheorem{prop}[equation]{Proposition}
\newtheorem{cor}[equation]{Corollary}

\theoremstyle{definition}
\newtheorem{defn}[equation]{Definition}
\newtheorem*{ack}{Acknowledgments}

\theoremstyle{remark}
\newtheorem{rem}[equation]{Remark}
\newtheorem{exam}[equation]{Example}

\numberwithin{equation}{section}

\newcommand{\ZZ}{\mathbb{Z}}
\newcommand{\QQ}{\mathbb{Q}}

\newcommand{\FF}{\mathbb{F}}

\newcommand{\bbS}{\mathbb{S}}

\newcommand{\bq}{\mathbf{q}}

\newcommand{\cD}{\mathcal{D}}
\newcommand{\cG}{\mathcal{G}}
\newcommand{\cH}{\mathcal{H}}
\newcommand{\cM}{\mathcal{M}}

\newcommand{\cU}{\mathcal{U}}

\newcommand{\fp}{\mathfrak{p}}

\DeclareMathOperator{\ch}{char}
\DeclareMathOperator{\GL}{GL}
\DeclareMathOperator{\lcm}{lcm}
\DeclareMathOperator{\ord}{ord}
\DeclareMathOperator{\SL}{SL}

\newcommand{\on}{\overline{n}}

\newcommand{\ui}{\underline{i}}

\newcommand{\pd}{\partial}
\newcommand{\C}{C_\infty}

\newcommand{\tpi}{\widetilde{\pi}}
\newcommand{\fhat}{\hat{f}}

\newcommand{\power}[2]{{#1 [[ #2 ]]}}
\newcommand{\laurent}[2]{{#1 (( #2 ))}}
\newcommand{\prodprime}{\sideset{}{'}{\prod}}
\newcommand{\sumprime}{\sideset{}{'}{\sum}}
\newcommand{\vnorm}[1]{\lVert #1 \rVert_v}

\begin{document}

\title[Theta operators, Goss polynomials, and $v$-adic modular forms]{Theta operators, Goss polynomials, \\ and $v$-adic modular forms}

\author{Matthew A. Papanikolas}
\address{Department of Mathematics, Texas A{\&}M University, College Station,
TX 77843, USA}
\email{papanikolas@tamu.edu}

\author{Guchao Zeng}
\address{Department of Mathematics, Texas A{\&}M University, College Station,
TX 77843, USA}
\email{zengguchao@math.tamu.edu}

\thanks{This project was partially supported by NSF Grant DMS-1501362}


\date{December 9, 2016}

\begin{abstract}
We investigate hyperderivatives of Drinfeld modular forms and determine formulas for these derivatives in terms of Goss polynomials for the kernel of the Carlitz exponential.  As a consequence we prove that $v$-adic modular forms in the sense of Serre, as defined by Goss and Vincent, are preserved under hyperdifferentiation.  Moreover, upon multiplication by a Carlitz factorial, hyperdifferentiation preserves $v$-integrality.
\end{abstract}


\dedicatory{In honor of David Goss}

\maketitle

\section{Introduction} \label{S:intro}

In \cite{Serre73}, Serre introduced $p$-adic modular forms for a fixed prime $p$, as $p$-adic limits of Fourier expansions of holomorphic modular forms on $\SL_2(\ZZ)$ with rational coefficients.  He established fundamental results about families of $p$-adic modular forms by developing the theories of differential operators and Hecke operators acting on $p$-adic spaces of modular forms, and in particular he showed that the weight $2$ Eisenstein series $E_2$ is also $p$-adic.  If we let $\vartheta := \frac{1}{2\pi i} \frac{d}{dz}$ be Ramanujan's theta operator acting on holomorphic complex forms, then letting $\bq(z) = e^{2\pi i z}$, we have
\begin{equation} \label{E:classicaltheta}
  \vartheta = \bq\frac{d}{d\bq}, \quad \vartheta(\bq^n) = n\bq^n.
\end{equation}
Although $\vartheta$ does not preserve spaces of complex modular forms, Serre proved the induced operation $\vartheta : \QQ \otimes \power{\ZZ_p}{\bq} \to \QQ \otimes \power{\ZZ_p}{\bq}$ does take $p$-adic modular forms to $p$-adic modular forms and preserves $p$-integrality.

In the present paper we investigate differential operators on spaces of $v$-adic modular forms, where $v$ is a finite place corresponding to a prime ideal of the polynomial ring $A = \FF_q[\theta]$, for $\FF_q$ a field with $q$ elements and $q$ itself a power of a prime number $p$.  Drinfeld modular forms were first studied by Goss~\cite{Goss80a}, \cite{Goss80b}, \cite{Goss80c}, as rigid analytic functions,
\[
  f : \Omega \to \C,
\]
on the Drinfeld upper half space $\Omega$ that transform with respect to the group $\Gamma = \GL_2(A)$ (see \S\ref{S:Theta} for precise definitions).  Here if we take $K = \FF_q(\theta)$, then $\Omega$ is defined to be $\C \setminus K_\infty$, where $K_\infty = \laurent{\FF_q}{1/\theta}$ is the completion of $K$ at its infinite place and $\C$ is the completion of an algebraic closure of $K_\infty$.  Goss showed that Drinfeld modular forms have expansions in terms of the uniformizing parameter $u(z) := 1/e_C(\tpi z)$ at the infinite cusp of $\Omega$, where $e_C(z)$ is the exponential function of the Carlitz module and $\tpi$ is the Carlitz period.  Each such form $f$ is uniquely determined by its $u$-expansion,
\[
  f = \sum_{n=0}^{\infty} c_n u^n  \in \power{\C}{u}.
\]
If $k \equiv 0 \pmod{q-1}$, then the weight $k$ Eisenstein series of Goss~\cite{Goss80b}, has a $u$-expansion due to Gekeler~\cite[(6.3)]{Gekeler88} of the form
\[
  E_k = -\frac{\zeta_C(k)}{\tpi^k} - \sum_{a \in A,\,a\,\textup{monic}}
  G_k( u(az)), \quad
  \frac{\zeta_C(k)}{\tpi^k} \in K,
\]
where $\zeta_C(k)$ is a Carlitz zeta value, $G_k(u)$ is a Goss polynomial of degree $k$ for the lattice $\Lambda_C = A\tpi$ (see \S\ref{S:Gosspolys}--\ref{S:Theta} and \eqref{E:gk}), and $u(az)$ can be shown to be represented as a power series in $u$ (see \S\ref{S:Theta}).  Gekeler and Goss also show that spaces of forms for $\Gamma$ are generated by forms with $u$-expansions with coefficients in~$A$.  Using this as a starting point, Goss~\cite{Goss14} and Vincent~\cite{Vincent14} defined $v$-adic modular forms in the sense of Serre by taking $v$-adic limits of $u$-expansions and thus defining $v$-adic forms as power series in $K \otimes_A \power{A_v}{u}$ (see \S\ref{S:vadicforms}).  Goss~\cite{Goss14} constructed a family of $v$-adic forms based on forms with $A$-expansions due to Petrov~\cite{Petrov13} (see Theorem~\ref{T:PetrovGoss}), and Vincent~\cite{Vincent14} showed that forms for the group $\Gamma_0(v) \subseteq \GL_2(A)$ with $v$-integral $u$-expansions are also $v$-adic modular forms.

It is natural to ask how Drinfeld modular forms and $v$-adic forms behave under differentiation, and since we are in positive characteristic it is favorable to use hyperdifferential operators $\pd_z^r$, rather than straight iteration $\frac{d^r}{dz^r} = \frac{d}{dz} \circ \cdots \circ \frac{d}{dz}$ (see \S\ref{S:functions} for definitions).  Gekeler~\cite[\S 8]{Gekeler88} showed that if we define $\Theta := -\frac{1}{\tpi} \frac{d}{dz} = -\frac{1}{\tpi}\, \pd_z^1$, then we have the action on $u$-expansions determined by the equality
\begin{equation} \label{E:Theta1}
  \Theta = u^2 \frac{d}{du} = u^2\, \pd_u^1.
\end{equation}
Now as in the classical case, derivatives of Drinfeld modular forms are not necessarily modular, but Bosser and Pellarin~\cite{BosserPellarin08}, \cite{BosserPellarin09}, showed that hyperdifferential operators $\pd_z^r$ preserve spaces of quasi-modular forms, i.e., spaces generated by modular forms and the false Eisenstein series $E$ of Gekeler (see Example~\ref{Ex:falseE}), which itself plays the role of $E_2$.

For $r \geqslant 0$, following Bosser and Pellarin we define the operator $\Theta^r$ by
\[
  \Theta^r := \frac{1}{(-\tpi)^r}\, \pd_z^r.
\]
Uchino and Satoh~\cite[Lem.~3.6]{UchinoSatoh98} proved that $\Theta^r$ takes functions with $u$-expansions to functions with $u$-expansions, and Bosser and Pellarin~\cite[Lem.~3.5]{BosserPellarin08} determined formulas for the expansion of $\Theta^r(u^n)$.  If we consider the $r$-th iterate of the classical $\vartheta$-operator, $\vartheta^{\circ r} = \vartheta \circ \cdots \circ \vartheta$, then clearly by \eqref{E:classicaltheta},
\[
  \vartheta^{\circ r} (q^n) = n^r q^n.
\]
If we iterate $\Theta$, taking $\Theta^{\circ r} = \Theta \circ \cdots \circ \Theta$, then by \eqref{E:Theta1} we find
\[
  \Theta^{\circ r}(u^n) = r! \binom{n+r-1}{r} u^{n+r},
\]
which vanishes identically when $r \geqslant p$.  On the other hand, the factor of $r!$ is not the only discrepancy in comparing $\Theta^r$ and $\Theta^{\circ r}$, and in fact we prove two formulas in Corollary~\ref{C:Thetau} revealing that $\Theta^r$ is intertwined with Goss polynomials for $\Lambda_C$:
\begin{align}
\label{E:ThetaA}
  \Theta^r(u^n) &= u^n\, \pd_u^{n-1} \bigl( u^{n-2} G_{r+1}(u) \bigr), \quad \forall\,n \geqslant 1, \\
\label{E:ThetaB}
  \Theta^r(u^n) &= \sum_{j=0}^r \binom{n+j-1}{j} \beta_{r,j} u^{n+j}, \quad \forall\, n \geqslant 0,
\end{align}
where $\beta_{r,j}$ are the coefficients of $G_{r+1}(u)$.  These formulas arise from general results (Theorem~\ref{T:Gosspolydiffs}) on hyperderivatives of Goss polynomials for arbitrary $\FF_q$-lattices in $\C$, which is the primary workhorse of this paper, and they induce formulas for hyperderivatives of $u$-expansions of Drinfeld modular forms (Corollary~\ref{C:Thetaf}).  It is important to note that~\eqref{E:ThetaB} is close to a formula of Bosser and Pellarin~\cite[Eq.~(28)]{BosserPellarin08}, although the connections with coefficients of Goss polynomials appears to be new and the approaches are somewhat different.

Goss~\cite{Goss14} defines the weight space of $v$-adic modular forms to be $\bbS = \ZZ/(q^d-1)\ZZ \times \ZZ_p$, where $d$ is the degree of $v$, and if we take $\cM_s^m \subseteq K \otimes_A \power{A_v}{u}$ to be the space of $v$-adic forms of weight $s \in \bbS$ and type $m \in \ZZ/(q-1)\ZZ$ (see \S\ref{S:vadicforms}), then we prove (Theorem~\ref{T:Thetavadic}) that $\Theta^r$ preserves spaces of $v$-adic modular forms,
\[
  \Theta^r : \cM_s^m \to \cM_{s+2r}^{m+r}, \quad r \geqslant 0.
\]
Of particular importance here is proving that the false Eisenstein series $E$ is a $v$-adic form (Theorem~\ref{T:Evadic}).    Vincent~\cite[Thm.~1.2]{Vincent10} showed that $\Theta(f)$ is congruent to a modular form modulo $v$, so we show that this congruence lifts to $v$-adic modular forms.  For $r\geqslant q$, unlike in the classical case, $\Theta^r$ does not preserve $v$-integrality due to denominators coming from $G_{r+1}(u)$, but we show in \S\ref{S:vintegral} that this failure can be controlled, namely showing (Theorem~\ref{T:vintegral}) that
\[
  \Pi_r \Theta^r : \cM_s^m(A_v) \to \cM_{s+2r}^{m+r}(A_v), \quad r \geqslant 0,
\]
where $\Pi_r \in A$ is the Carlitz factorial (see \S\ref{S:functions}) and $\cM_s^m(A_v) = \cM_s^m \cap \power{A_v}{u}$.

\begin{ack}
The authors thank David Goss for a number of helpful suggestions on an earlier version of this article, and they thank Bruno Angl\`es for translating the abstract into French.  The authors also thank the referee for useful suggestions.
\end{ack}

\section{Functions and hyperderivatives} \label{S:functions}

Let $\FF_q$ be the finite field with $q$ elements, $q$ a fixed power of a prime $p$.  Let $A := \FF_q[\theta]$ be a polynomial ring in one variable, and let $K := \FF_q(\theta)$ be its fraction field.  We let $A_+$ denote the monic elements of $A$, $A_{d+}$ the monic elements of degree $d$, and $A_{(< d)}$ the elements of $A$ of degree $< d$.

For each place $v$ of $K$, we define an absolute value $|\cdot|_v$ and valuation $\ord_v$, normalized in the following way.  If $v$ is a finite place, we fix $\wp \in A_+$ to be the monic generator of the prime ideal $\fp_v$ corresponding to $v$ and we set $|\wp|_v = 1/q^{\deg \wp}$ and $\ord_v(\wp) = 1$.  If $v = \infty$, then we set $|\theta|_\infty = q$ and $\ord_\infty(\theta) = -\deg(\theta) = -1$.  For any place $v$ we let $A_v$ and $K_v$ denote the $v$-adic completions of $A$ and $K$.  For the place $\infty$, we note that $K_\infty = \laurent{\FF_q}{1/\theta}$, and we let $\C$ be a completion of an algebraic closure of $K_\infty$.  Finally, we let $\Omega := \C \setminus K_\infty$ be the Drinfeld upper half-plane of $\C$.

For $i \geqslant 1$, we set
\begin{equation}
  [i] = \theta^{q^i} - \theta, \quad D_i = [i] [i-1]^{q} \cdots [1]^{q^{i-1}}, \quad
  L_i = (-1)^i[i][i-1] \cdots [1],
\end{equation}
and we let $D_0=L_0=1$.  We have the recursions, $D_i = [i]D_{i-1}^q$ and $L_i = -[i] L_{i-1}$, and we recall~\cite[Prop.~3.1.6]{Goss} that
\begin{equation}
  [i] = \prod_{\substack{f \in A_+,\ \textup{irred.} \\ \deg(f) \mid i}} f,
  \quad D_i = \prod_{a \in A_{i+}} a,
  \quad L_i = (-1)^i \cdot \lcm(f \in A_{i+}).
\end{equation}
For $m \in \ZZ_+$, we define the Carlitz factorial $\Pi_m$ as follows.  If we write $m = \sum m_i q^i$ with $0 \leqslant m_i \leqslant q-1$, then
\begin{equation} \label{E:Pi}
  \Pi_m = \prod_i D_i^{m_i}.
\end{equation}
For more information about $\Pi_m$ the reader is directed to Goss~\cite[\S 9.1]{Goss}.

For an $\FF_q$-algebra $L$, we let $\tau : L \to L$ denote the $q$-th power Frobenius map, and we let $L[\tau]$ denote the ring of twisted polynomials over $L$, subject to the condition that $\tau c=c^q \tau$ for $c \in L$.  We then define as usual the Carlitz module to be the $\FF_q$-algebra homomorphism $C : A \to A[\tau]$ determined by
\[
  C_\theta = \theta + \tau.
\]
The Carlitz exponential is the $\FF_q$-linear power series,
\begin{equation} \label{E:elogdef}
  e_C(z) = \sum_{i=0}^\infty \frac{z^{q^i}}{D_i}.
\end{equation}
The induced function $e_C : \C \to \C$ is both entire and surjective, and for all $a \in A$,
\[
  e_C(az) = C_a(e_C(z)).
\]
The kernel $\Lambda_C$ of $e_C(z)$ is the $A$-lattice of rank~$1$ given by $\Lambda_C = A\tpi$, where for a fixed $(q-1)$-st root of $-\theta$,
\[
  \tpi = \theta (-\theta)^{1/(q-1)} \prod_{i=1}^\infty \Bigl( 1 - \theta^{1-q^i} \Bigr)^{-1} \in K_{\infty}\bigl( (-\theta)^{1/(q-1)} \bigr)
\]
is called the Carlitz period (see~\cite[\S 3.2]{Goss} or \cite[\S 3.1]{PLogAlg}).  Moreover, we have a product expansion
\begin{equation} \label{E:eCprod}
  e_C(z) = z \prodprime_{\lambda \in \Lambda_C} \biggl( 1 - \frac{z}{\lambda} \biggr) = z \sideset{}{'}{\prod}_{a \in A} \biggl( 1 - \frac{z}{a\tpi} \biggr),
\end{equation}
where the prime indicates omitting the $a=0$ term in the product.  For more information about the Carlitz module, and Drinfeld modules in general, we refer the reader to~\cite[Chs.~3--4]{Goss}.

We will say that a function $f : \Omega \to \C$ is holomorphic if it is rigid analytic in the sense of~\cite{FresnelvdPut}.  We set $\cH(\Omega)$ to be the set of holomorphic functions on $\Omega$.  We define a holomorphic function $u : \Omega \to \C$ by setting
\begin{equation} \label{E:udef}
  u(z): = \frac{1}{e_C(\tpi z)},
\end{equation}
and we note that $u(z)$ is a uniformizing parameter at the infinite cusp of $\Omega$ (see \cite[\S 5]{Gekeler88}), which plays the role of $\bq(z) = e^{2 \pi i z}$ in the classical case.   The function $u(z)$ is $A$-periodic in the sense that $u(z+a) = u(z)$ for all $a \in A$.  The imaginary part of an element $z \in \C$ is set to be
\[
  |z|_i = \inf_{x \in K_\infty} |z-x|_\infty,
\]
which measures the distance from $z$ to the real axis $K_\infty \subseteq \C$.  We will say that an $A$-periodic holomorphic function $f : \Omega \to \C$ is holomorphic at $\infty$ if we can write a convergent series,
\[
  f(z) = \sum_{n=0}^\infty c_n u(z)^n, \quad c_n \in \C, \quad |z|_i \gg 0.
\]
The function $f$ is then determined by the power series $f = \sum c_n u^n \in \power{\C}{u}$, and we call this power series the $u$-expansion of $f$ and the coefficients $c_n$ the $u$-expansion coefficients of $f$.  We set $\cU(\Omega)$ to be the subset of $\cH(\Omega)$ comprising functions on $\Omega$ that are $A$-periodic and holomorphic at $\infty$.  In other words, $\cU(\Omega)$ consists of functions that have $u$-expansions.

We now define hyperdifferential operators and hyperderivatives (see \cite{Conrad00}, \cite{Jeong11}, \cite{UchinoSatoh98} for more details).  For a field $F$ and an independent variable $z$ over $F$, for $j \geqslant 0$ we define the $j$-th hyperdifferential operator $\pd_z^j : F[z] \to F[z]$ by setting
\[
  \pd_z^j(z^n) = \binom{n}{j} z^{n-j}, \quad n \geqslant 0,
\]
where $\binom{n}{j} \in \ZZ$ is the usual binomial coefficient, and extending $F$-linearly.  (By usual convention $\binom{n}{j} = 0$ if $0 \leqslant n < j$.)  For $f \in F[z]$, we call $\pd_z^j(f) \in F[z]$ its $j$-th hyperderivative.  Hyperderivatives satisfy the product rule,
\begin{equation} \label{E:prod}
  \pd_z^j (fg) = \sum_{k=0}^j \pd_z^k(f)\pd_z^{j-k}(g), \quad f, g \in F[z],
\end{equation}
and composition rule,
\begin{equation} \label{E:comp}
  (\pd_z^j \circ \pd_z^k)(f) = (\pd_z^k \circ \pd_z^j)(f) = \binom{j+k}{j} \pd_z^{j+k}(f), \quad f \in F[z].
\end{equation}
Using the product rule one can extend to $\pd_z^j: F(z) \to F(z)$ in a unique way, and $F(z)$ together with the operators $\pd_z^j$ form a hyperdifferential system.  If $F$ has characteristic~$0$, then $\pd_z^j = \frac{1}{j!} \frac{d^j}{dz^j}$, but in characteristic $p$ this holds only for $j \leqslant p-1$.  Furthermore, hyperderivatives satisfy a number of differentiation rules (e.g., product, quotient, power, chain rules), which aid in their description and calculation  (see~\cite[\S 2.2]{Jeong11}, \cite[\S 2.3]{PLogAlg}, for a complete list of rules and historical accounts).  Moreover, if $f \in F(z)$ is regular at $c \in F$, then so is $\pd_z^j(f)$ for each $j \geqslant 0$, and it follows that we have a Taylor expansion,
\begin{equation} \label{E:Taylor}
  f(z) = \sum_{j=0}^\infty \pd_z^j(f)(c) \cdot (z-c)^j \in \power{F}{z-c}.
\end{equation}
In this way we can also extend $\pd_z^j$ uniquely to $\pd_z^j : \laurent{F}{z-c} \to \laurent{F}{z-c}$.

For a holomorphic function $f : \Omega \to \C$, it was proved by Uchino and Satoh~\cite[\S 2]{UchinoSatoh98} that we can define a holomorphic hyperderivative $\pd_z^j(f) : \Omega \to \C$ (taking $F = \C$ in the preceding paragraph).  That is,
 \[
  \pd_z^j : \cH(\Omega) \to \cH(\Omega).
 \]
Moreover they prove that the system of operators $\pd_z^j$ on holomorphic functions inherits the same differentiation rules for hyperderivatives of polynomials and power series.  Thus for $f \in \cH(\Omega)$ and $c \in \Omega$, we have a Taylor expansion,
\[
  f(z) = \sum_{j=0}^\infty \pd_z^j(f)(c) \cdot (z-c)^j \in \power{\C}{z-c}.
\]
We have the following crucial lemma for our later considerations in \S\ref{S:Theta}, where we find new identities for derivatives of functions in $\cU(\Omega)$.

\begin{lem}[{Uchino-Satoh \cite[Lem.~3.6]{UchinoSatoh98}}]
If $f \in \cH(\Omega)$ is $A$-periodic and holomorphic at $\infty$, then so is $\pd_z^j(f)$ for each $j \geqslant 0$.  That is,
\[
\pd_z^j : \cU(\Omega) \to \cU(\Omega), \quad j \geqslant 0.
\]
\end{lem}

We recall computations involving $u(z)$ and $\pd_z^1(u(z))$ (see \cite[\S 3]{Gekeler88}).  First we see from~\eqref{E:elogdef} that $\pd_z^1(e_C(z)) = 1$, so using \eqref{E:eCprod} and taking logarithmic derivatives,
\begin{equation} \label{E:tsum}
  u(z) = \frac{1}{e_C(\tpi z)} = \frac{1}{\tpi} \cdot \frac{ \pd_z^1(e_C(\tpi z))}{e_C(\tpi z)}
  = \frac{1}{\tpi} \sum_{a \in A} \frac{1}{z+a}.
\end{equation}
Furthermore,
\begin{equation} \label{E:uder1z}
  \pd_z^1(u(z)) = \pd_z^1 \biggl( \frac{1}{e_C(\tpi z)} \biggr) = \frac{-\pd_z^1(e_C(\tpi z))}{e_C(\tpi z)^2} = -\tpi u(z)^2.
\end{equation}
Thus, $\pd_z^1(u) = -\tpi u^2 \in \cU(\Omega)$.  In \S\ref{S:Theta} we generalize this formula and calculate $\pd_z^r(u^n)$ for $r$, $n \geqslant 0$.

We conclude this section by discussing some properties of hyperderivatives particular to positive characteristic.  Suppose $\ch(F) = p >0$.  If we write $j = \sum_{i=0}^s b_i p^i$, with $0 \leqslant b_i \leqslant p-1$ and $b_s \neq 0$, then (see \cite[Thm.~3.1]{Jeong11})
\begin{equation} \label{E:hyppthpowers}
  \pd_z^j = \pd_z^{b_0} \circ \pd_z^{b_1 p} \circ \cdots
  \circ \pd_z^{b_s p^s},
\end{equation}
which follows from the composition law and  Lucas's theorem (e.g., see \cite[Eq.~(14)]{BosserPellarin08}).  We note that for $0 \leqslant b \leqslant p-1$,
\[
  \pd_z^{bp^k} = \frac{1}{b!} \cdot \pd_z^{p^k} \circ \cdots \circ \pd_z^{p^k}, \quad \textup{($b$ times).}
\]
Moreover the $p$-th power rule (see \cite[\S 7]{Brownawell99}, \cite[\S 2.2]{Jeong11}) says that for $f \in \laurent{F}{z-c}$,
\begin{equation} \label{E:hyppthpowers2}
  \pd_z^j\bigl( f^{p^s} \bigr) =
  \begin{cases}
  \bigl( \pd_z^{\ell}(f) \bigr)^{p^s} & \textup{if $j = \ell p^s$,} \\
  0 & \textup{otherwise,}
  \end{cases}
\end{equation}
and so calculation using \eqref{E:hyppthpowers} and \eqref{E:hyppthpowers2} can often be fairly efficient.

\section{Goss polynomials and hyperderivatives} \label{S:Gosspolys}

We review here results on Goss polynomials, which were introduced by Goss in~\cite[\S 6]{Goss80c} and have been studied further by Gekeler~\cite[\S 3]{Gekeler88}, \cite{Gekeler13}.  We start first with an $\FF_q$-vector space $\Lambda \subseteq \C$ of dimension $d$.  We define the exponential function of~$\Lambda$,
\[
  e_\Lambda(z) = z \prodprime_{\lambda \in \Lambda} \biggl( 1 - \frac{z}{\lambda} \biggr),
\]
which is an $\FF_q$-linear polynomial of degree~$q^d$.  If we take $t_\Lambda(z) = 1/e_{\Lambda}(z)$, then just as in~\eqref{E:tsum} we have
\[
t_{\Lambda}(z) = \sum_{\lambda \in \Lambda} \frac{1}{z-\lambda}.
\]
We can extend these definitions to any discrete lattice $\Lambda \subseteq \C$, which is the union of nested finite dimensional $\FF_q$-vector spaces $\Lambda_1 \subseteq \Lambda _2 \subseteq \cdots$.  We find that generally $e_{\Lambda}(z) = \lim_{i \to \infty} e_{\Lambda_i}(z)$ and $t_{\Lambda}(z) = \lim_{i \to \infty} t_{\Lambda_i}(z)$, where the convergence is coefficient-wise in $\laurent{\C}{z}$.

\begin{rem} \label{R:uvst}
If we take $\Lambda = \Lambda_C$, then $e_{\Lambda_C}(z) = e_C(z)$, whereas if we take $\Lambda = A$, then $e_{A}(z) = \frac{1}{\tpi} e_C(\tpi z)$.  Thus
\[
  t_A(z) = \tpi\, t_{\Lambda_C}(\tpi z) = \frac{\tpi}{e_C(\tpi z)},
\]
and $u(z)$, as defined in~\eqref{E:udef}, is given by
\[
  u(z) = \frac{t_A(z)}{\tpi} = t_{\Lambda_C}(\tpi z).
\]
This normalization of $u(z)$ is taken so that the $u$-expansions of some Drinfeld modular forms will have $K$-rational coefficients.
\end{rem}

\begin{thm}[{Goss \cite[\S 6]{Goss80c}; see also Gekeler~\cite[\S 3]{Gekeler88}}] \label{T:Gosspolys}
Let $\Lambda \subseteq \C$ be a discrete $\FF_q$-vector space.  Let
\[
e_{\Lambda}(z) = z \prodprime_{\lambda \in \Lambda} \biggl( 1 - \frac{z}{\lambda} \biggr) = \sum_{j=0}^\infty \alpha_j z^{q^j},
\]
and let $t_{\Lambda}(z) = 1/e_{\Lambda}(z)$.  For each $k \geqslant 1$, there is a monic polynomial $G_{k,\Lambda}(t)$ of degree $k$ with coefficients in $\FF_q[\alpha_0, \alpha_1, \ldots]$ so that
\[
  S_{k,\Lambda}(z) := \sum_{\lambda \in \Lambda} \frac{1}{(z-\lambda)^k} = G_{k,\Lambda}\bigl(
  t_{\Lambda}(z) \bigr).
\]
Furthermore the following properties hold.
\begin{enumerate}
\item[(a)] $G_{k,\Lambda}(t) = t (G_{k-1,\Lambda}(t) + \alpha_1 G_{k-q,\Lambda}(t) + \alpha_2 G_{k-q^2,\Lambda}(t) + \cdots )$.
\item[(b)] We have a generating series identity
\[
  \cG_{\Lambda}(t,x) = \sum_{k=1}^\infty G_{k,\Lambda}(t) x^k = \frac{tx}{1 - t e_{\Lambda}(x)}.
\]
\item[(c)] If $k \leqslant q$, then $G_{k,\Lambda}(t) = t^k$.
\item[(d)] $G_{pk,\Lambda}(t) = G_{k,\Lambda}(t)^p$.
\item[(e)] $t^2\, \pd_t^1 \bigl( G_{k,\Lambda}(t) \bigr) = k G_{k+1,\Lambda}(t)$.
\end{enumerate}
\end{thm}

Gekeler~\cite[(3.8)]{Gekeler88} finds a formula for each $G_{k,\Lambda}(t)$,
\begin{equation} \label{E:Gossexpansion}
  G_{k+1,\Lambda}(t) = \sum_{j=0}^k \sum_{\ui} \binom{j}{\ui} \alpha^{\ui} t^{j+1},
\end{equation}
where the sum is over all $(s+1)$-tuples $\ui = (i_0, \dots, i_s)$, with $s$ arbitrary, satisfying $i_0 + \cdots + i_s = j$ and $i_0 + i_1q + \cdots + i_sq^s = k$; $\binom{j}{\ui} = j!/(i_0! \cdots i_s!)$ is a multinomial coefficient; and $\alpha^{\ui} = \alpha_0^{i_0} \cdots \alpha_s^{i_s}$.

Part~(e) of Theorem~\ref{T:Gosspolys} indicates that there are interesting hyperderivative relations among Goss polynomials, with respect to $t$ and to $z$, which we now investigate.  All hyperderivatives we will take will be of polynomials and formal power series, but the considerations in \S\ref{S:functions} about holomorphic functions will play out later in the paper.  The main result of this section is the following.

\begin{thm} \label{T:Gosspolydiffs}
Let $\Lambda \subseteq \C$ be a discrete $\FF_q$-vector space, and let $t = t_{\Lambda}(z)$.  For $r \geqslant 0$, we define $\beta_{r,j}$ so that
\[
  G_{r+1,\Lambda}(t) = \sum_{j=0}^{r} \beta_{r,j} t^{j+1}.
\]
Then
\begin{equation} \label{E:tndiff}
\tag{\ref{T:Gosspolydiffs}a}
\pd_z^r ( t^n ) = (-1)^r \cdot t^n\, \pd_t^{n-1} \bigl( t^{n-2} G_{r+1,\Lambda}(t) \bigr), \quad \forall\, n \geqslant 1,
\end{equation}
\begin{equation} \label{E:tndiffbeta}
\tag{\ref{T:Gosspolydiffs}b}
\pd_z^r (t^n) = (-1)^r \sum_{j=0}^r \beta_{r,j}\, t^{j+1}\, \pd_t^{j} (t^{n+j-1}), \quad \forall\, n \geqslant 0,
\end{equation}
and
\begin{equation} \label{E:Gossdiff}
\tag{\ref{T:Gosspolydiffs}c}
\binom{n+r-1}{r} G_{n+r,\Lambda}(t) = \sum_{j=0}^{r}
\beta_{r,j}\, t^{j+1}\, \pd_t^{j} \bigl( t^{j-1} G_{n,\Lambda}(t) \bigr),\quad \forall\, n \geqslant 1.
\end{equation}
\end{thm}

\begin{rem}
We see that \eqref{E:tndiff} and~\eqref{E:tndiffbeta} generalize~\eqref{E:uder1z} and that~\eqref{E:Gossdiff} generalizes Theorem~\ref{T:Gosspolys}(e).  In later sections \eqref{E:tndiff} and~\eqref{E:tndiffbeta} will be useful for taking derivatives of Drinfeld modular forms.  The coefficients $\beta_{r,j}$ can be computed using the generating series $\cG_{\Lambda}(t,x)$ or equivalently~\eqref{E:Gossexpansion}.  The proof requires some preliminary lemmas.
\end{rem}

\begin{lem}[{cf.~Petrov \cite[\S 3]{Petrov15}}] \label{L:Lem1}
For $r \geqslant 0$ and $n \geqslant 1$,
\begin{equation} \label{E:diffrSnG}
\tag{\ref{L:Lem1}a}
\pd_z^r\bigl( S_{n,\Lambda}(z) \bigr) = (-1)^r \binom{n+r-1}{r} G_{n+r,\Lambda}(t).
\end{equation}
Moreover, we have
\begin{equation} \label{E:diffrSndual}
\tag{\ref{L:Lem1}b}
\pd_z^r \bigl( S_{n,\Lambda}(z) \bigr) = (-1)^{n+r-1} \cdot \pd_z^{n-1} \bigl( S_{r+1,\Lambda}(z) \bigr)
\end{equation}
and
\begin{equation} \label{E:diffrt}
\tag{\ref{L:Lem1}c}
\pd_z^r(t) = (-1)^r G_{r+1,\Lambda}(t).
\end{equation}
\end{lem}

\begin{proof}
First of all, we recall the convention that for $n>0$ and $r \geqslant 0$, we have $\binom{-n}{r} = (-1)^r\binom{n+r-1}{r}$.  Then using the power and quotient rules~\cite[\S 2.2]{Jeong11}, we see that for $\lambda \in \C$,
\[
  \pd_z^r \biggl( \frac{1}{(z-\lambda)^n} \biggr) = \binom{-n}{r} \frac{1}{(z-\lambda)^{n+r}}
  = (-1)^r \binom{n+r-1}{r} \frac{1}{(z-\lambda)^{n+r}}.
\]
Therefore,
\[
  \pd_z^r\bigl( S_{n,\Lambda}(z) \bigr) = (-1)^r \binom{n+r-1}{r} S_{n+r,\Lambda}(z),
\]
and combining with the defining property of $G_{n+r,\Lambda}(t)$ in Theorem~\ref{T:Gosspolys}, we see that~\eqref{E:diffrSnG} follows.  Now
\[
  \binom{n+r-1}{r} = \binom{(r+1) + (n-1) - 1}{n-1},
\]
and so~\eqref{E:diffrSndual} follows from~\eqref{E:diffrSnG}.  Finally, \eqref{E:diffrt} is a special case of \eqref{E:diffrSnG} with $n=1$.
\end{proof}

\begin{lem} \label{L:genfundiff}
For $n \geqslant 1$, we have an identity of rational functions in $x$,
\[
  \frac{x}{(1-t e_\Lambda(x))^n} = \pd_t^{n-1} \biggl( \frac{ t^{n-1} x}{1 - t e_\Lambda(x)} \biggr)
  = \pd_t^{n-1} \bigl( t^{n-2} \cG_{\Lambda}(t,x) \bigr).
\]
\end{lem}

\begin{proof}
Our derivatives with respect to $t$ are taken while considering $x$ to be a constant.  We note that for $\ell \geqslant 0$,
\[
  \pd_t^{\ell} \biggl( \frac{1}{1-t e_{\Lambda}(x) } \biggr) =
  \frac{e_\Lambda(x)^\ell}{(1-te_{\Lambda}(x))^{\ell+1}},
\]
by the quotient and chain rules~\cite[\S 2.2]{Jeong11}.  Therefore, by the product rule,
\begin{align*}
  \pd_t^{n-1} \biggl( \frac{t^{n-1}}{1- t e_{\Lambda}(x)} \biggr)
  &= \sum_{k=0}^{n-1} \pd_t^{k} (t^{n-1}) \pd_t^{n-1-k} \biggl( \frac{1}{1- te_{\Lambda}(x)} \biggr) \\
  &= \sum_{k=0}^{n-1} \binom{n-1}{k} \biggl( \frac{t e_{\Lambda}(x)}{1-t e_{\Lambda}(x)} \biggr)^{n-1-k} \cdot \frac{1}{1- te_{\Lambda}(x)} \\
  &= \biggl( 1 + \frac{t e_{\Lambda}(x)}{1 - t e_{\Lambda}(x)} \biggr)^{n-1} \cdot \frac{1}{1- te_{\Lambda}(x)}.
\end{align*}
A simple calculation yields that this is $1/(1-te_{\Lambda}(x))^n$, and the result follows.
\end{proof}

\begin{proof}[Proof of Theorem~\ref{T:Gosspolydiffs}]
The chain rule~\cite[\S 2.2]{Jeong11} and \eqref{E:diffrt} imply that
\begin{align*}
  \pd_z^{r}(t^n) &= \sum_{k=1}^r \binom{n}{k} t^{n-k}
  \sum_{\substack{\ell_1, \dots, \ell_{k} \geqslant 1 \\ \ell_1 + \cdots + \ell_k = r}}
  \pd_z^{\ell_1}(t) \cdots \pd_z^{\ell_k}(t) \\
  &= (-1)^r \sum_{k=1}^r \binom{n}{k} t^{n-k} \sum_{\substack{\ell_1, \dots, \ell_{k} \geqslant 1 \\ \ell_1 + \cdots + \ell_k = r}} G_{\ell_1 + 1,\Lambda}(t) \cdots G_{\ell_k+1,\Lambda}(t).
\end{align*}
By direct expansion (see~\cite[\S 2.2, Eq.~(I)]{Jeong11}), the final inner sum above is the coefficient of $x^r$ in
\[
  \bigl( G_{2,\Lambda}(t) x + G_{3,\Lambda}(t)x^2 + \cdots \bigr)^k,
\]
and therefore by the binomial theorem,
\[
  \pd_z^r(t^n) = (-1)^r \cdot \textup{$\Bigl($coefficient of $x^r$ in
  $\bigl( t + G_{2,\Lambda}(t) x + G_{3,\Lambda}(t) x^2 + \cdots \bigr)^n\Bigr)$. }
\]
Now $G_{1,\Lambda}(t) = t$, so
\[
  t+ G_{2,\Lambda}(t)x + G_{3,\Lambda}(t)x^2 + \cdots = \sum_{k=1}^\infty G_{k,\Lambda}(t) x^{k-1}
  = \frac{\cG_\Lambda(t,x)}{x} = \frac{t}{1-te_{\Lambda}(x)}.
\]
Therefore,
\[
  \pd_z^r(t^n) = (-1)^r \cdot \textup{$\biggl($coefficient of $x^{r+1}$ in $\dfrac{t^nx}{(1-t e_\Lambda(x))^n}\biggr)$.}
\]
{From} Lemma~\ref{L:genfundiff} we see that
\[
  \frac{t^nx}{(1-t e_\Lambda(x))^n} = t^n \sum_{k=1}^\infty \pd_t^{n-1} \bigl( t^{n-2} G_{k,\Lambda}(t)
  \bigr) x^k,
\]
and so~\eqref{E:tndiff} holds.  To prove~\eqref{E:tndiffbeta}, we first note that it holds when $n=0$ by checking the various cases and using that $\beta_{r,0} = 0$ for $r \geqslant 1$, since $G_{r+1,\Lambda}(t)$ is divisible by $t^2$ for $r \geqslant 1$ by Theorem~\ref{T:Gosspolys}, and that $\beta_{0,0} = 1$.  For $n \geqslant 1$, we use~\eqref{E:tndiff} and write
\[
  \pd_z^r(t^n) = (-1)^r\cdot t^n\, \pd_t^{n-1}\bigl(t^{n-2}G_{r+1,\Lambda}(t) \bigr)
  = (-1)^r\cdot t^n\,\pd_t^{n-1} \biggl( \sum_{j=0}^r \beta_{r,j}\,t^{n+j-1} \biggr).
\]
Noting that
\[
  \pd_t^{n-1} (t^{n+j-1}) = \binom{n+j-1}{n-1} t^j = t^{j-n+1}\,\pd_t^{j}(t^{n+j-1}),
\]
we then have
\[
  \pd_z^r(t^n) = (-1)^r \sum_{j=0}^r \beta_{r,j}\, t^{j+1}\, \pd_t^j(t^{n+j-1}),
\]
and so \eqref{E:tndiffbeta} holds.  Furthermore, by~\eqref{E:diffrSnG} and~\eqref{E:diffrSndual},
\[
  \binom{n+r-1}{r} G_{n+r,\Lambda}(t) = (-1)^{n-1} \cdot \pd_z^{n-1}\bigl( S_{r+1,\Lambda}(z) \bigr)
  = (-1)^{n-1} \cdot \pd_z^{n-1} \bigl( G_{r+1,\Lambda}(t) \bigr).
\]
But then by~\eqref{E:tndiff},
\[
  \pd_z^{n-1} \bigl( G_{r+1,\Lambda}(t) \bigr) = \sum_{j=0}^r \beta_{r,j} \pd_z^{n-1}(t^{j+1})
  = (-1)^{n-1} \sum_{j=0}^r \beta_{r,j} t^{j+1}\, \pd_t^j \bigl( t^{j-1} G_{n,\Lambda}(t) \bigr),
\]
which yields~\eqref{E:Gossdiff}.
\end{proof}

\section{Theta operators on Drinfeld modular forms} \label{S:Theta}

We recall the definition of Drinfeld modular forms for $\GL_2(A)$, which were initially studied by Goss~\cite{Goss80a}, \cite{Goss80b}, \cite{Goss80c}.  We will also review results on $u$-expansions of modular forms due to Gekeler~\cite{Gekeler88}.  Throughout we let $\Gamma = \GL_2(A)$.  A holomorphic function $f : \Omega \to \C$ is a Drinfeld modular form of weight $k \geqslant 0$ and type $m \in \ZZ/(q-1)\ZZ$ if
\begin{enumerate}
\item for all $\gamma  = \left( \begin{smallmatrix} a & b \\ c & d \end{smallmatrix} \right) \in \Gamma$ and all $z \in \Omega$,
\[
  f(\gamma z) = (\det \gamma)^{-m} (cz+d)^k f(z), \quad \gamma z = \frac{az+b}{cz+d};
\]
\item and $f$ is holomorphic at $\infty$, i.e., $f$ has a $u$-expansion and so $f \in \cU(\Omega)$.
\end{enumerate}
We let $M_k^m$ be the $\C$-vector space of modular forms of weight $k$ and type $m$.  We know that $M_k^m \cdot M_{k'}^{m'} \subseteq M_{k+k'}^{m+m'}$ and that $M = \bigoplus_{k,m} M_k^m$ and $M^0 = \bigoplus_k M_k^0$ are graded $\C$-algebras.  Moreover, in order to have $M_k^m \neq 0$, we must have $k \equiv 2m \pmod{q-1}$.  If $L$ is a subring of $\C$, then we let $M_k^m(L)$ denote the space of forms with $u$-expansion coefficients in $L$, i.e., $M_k^m(L) = M_k^m \cap \power{L}{u}$.  We note that if $f = \sum c_n u^n$ is the $u$-expansion of $f \in M_k^m$, then
\begin{equation} \label{E:exponents}
  c_n \neq 0 \quad \Rightarrow \quad n \equiv m \pmod{q-1},
\end{equation}
which can be seen by using $\gamma = \left( \begin{smallmatrix} \zeta & 0 \\ 0 & 1 \end{smallmatrix} \right)$, for $\zeta$ a generator of $\FF_q^{\times}$, in the definition above.

Certain Drinfeld modular forms can be expressed in terms of $A$-expansions, which we now recall.  For $k \geqslant 1$, we set
\begin{equation}\label{E:gk}
  G_k(t) = G_{k,\Lambda_C}(t) = \sum_{j=0}^{k-1} \beta_{k-1,j} t^{j+1},
\end{equation}
to be the Goss polynomials with respect to the lattice $\Lambda_C$.  Since $e_C(z) \in \power{K}{z}$, it follows from Theorem~\ref{T:Gosspolys} that the coefficients $\beta_{k-1,j} \in K$ for all $k$, $j$.  As in~\eqref{E:udef} and Remark~\ref{R:uvst}, we have $u(z) = 1/e_C(\tpi z)$, and for $a \in A$ we set
\begin{equation} \label{E:uadef}
  u_a(z) := u(az) = \frac{1}{e_C(\tpi a z)}.
\end{equation}
Since $e_C(\tpi a z) = C_a(e_C(\tpi z))$, if we take the reciprocal polynomial for $C_a(z)$ to be $R_a(z) = z^{q^{\deg a}} C_a( 1/z)$
then
\begin{equation} \label{E:uasum}
  u_a = \frac{u^{q^{\deg a}}}{R_a(u)} = u^{q^{\deg a}} + \cdots \in \power{A}{u}.
\end{equation}
We say that a modular form $f$ has an $A$-expansion if there exist $k \geqslant 1$ and $c_0$, $c_a \in \C$ for $a \in A_+$, so that
\[
  f = c_0 + \sum_{a \in A_+} c_a G_k(u_a).
\]

\begin{exam}
For $k \equiv 0 \pmod{q-1}$, $k > 0$, the primary examples of Drinfeld modular forms with $A$-expansions come from Eisenstein series,
\[
  E_k(z) = \frac{1}{\tpi^k} \sumprime_{a,b \in A} \frac{1}{(az+b)^k},
\]
which is a modular form of weight $k$ and type $0$.  Gekeler~\cite[(6.3)]{Gekeler88} showed that
\begin{equation} \label{E:Eisexp}
  E_k = \frac{1}{\tpi^k} \sumprime_{b \in A} \frac{1}{b^k} - \sum_{a \in A_+} G_k(u_a)
  = -\frac{\zeta_C(k)}{\tpi^k} - \sum_{a \in A_+} G_k(u_a),
\end{equation}
where $\zeta_C(k) = \sum_{a \in A_+} a^{-k}$ is a Carlitz zeta value.  We know (see~\cite[\S 9.2]{Goss}) that $\zeta_C(k)/\tpi^k \in K$.
\end{exam}

For more information and examples on $A$-expansions the reader is directed to Gekeler \cite{Gekeler88}, L\'{o}pez \cite{Lopez10}, \cite{Lopez11}, and Petrov~\cite{Petrov13}, \cite{Petrov15}.

\begin{exam} \label{Ex:falseE}
We can also define the false Eisenstein series $E(z)$ of Gekeler~\cite[\S 8]{Gekeler88} to be
\[
  E(z) := \frac{1}{\tpi} \sum_{a \in A_+} \sum_{b \in A} \frac{a}{az+b},
\]
which is not quite a modular form but is a quasi-modular form similar to the classical weight~$2$ Eisenstein series~\cite{BosserPellarin08}, \cite{Gekeler88}.  Gekeler showed that $E \in \cU(\Omega)$ and that $E$ has an $A$-expansion,
\begin{equation} \label{E:falseexp}
  E = \sum_{a \in A_+} a G_{1}(u_a) = \sum_{a \in A_+} a u_a.
\end{equation}
\end{exam}

We now define theta operators $\Theta^r$ on functions in $\cH(\Omega)$ by setting for $r \geqslant 0$,
\begin{equation} \label{E:Thetadef}
  \Theta^r := \frac{1}{(-\tpi)^r}\, \pd_z^r.
\end{equation}
If we take $\Theta = \Theta^1$, then by~\eqref{E:uder1z}, $\Theta u = u^2$,
and $\Theta$ plays the role of the classical theta operator $\vartheta = \bq \frac{d}{d\bq}$.  Just as in the classical case, $\Theta$ and more generally $\Theta^r$ do not take modular forms to modular forms.  However, Bosser and Pellarin~\cite[Thm.~2]{BosserPellarin08} prove that $\Theta^r$ preserves quasi-modularity:
\[
  \Theta^r : \C[E,g,h] \to \C[E,g,h],
\]
where $E$ is the false Eisenstein series, $g = E_{q-1}$, and $h$ is the cusp form of weight $q+1$ and type $1$ defined by Gekeler~\cite[Thm.~5.13]{Gekeler88} as the $(q-1)$-st root of the discriminant function $\Delta$.  To prove their theorem, Bosser and Pellarin~\cite[Lem.~3.5]{BosserPellarin08} give formulas for $\Theta^r(u^n)$, which are ostensibly a bit complicated. From Theorem~\ref{T:Gosspolydiffs}, we have the following corollary, which perhaps conceptually simplifies matters.

\begin{cor} \label{C:Thetau}
For $r \geqslant 0$,
\begin{align}
\label{E:Thetaun}
\tag{\ref{C:Thetau}a}
  \Theta^r(u^n) &= u^n\, \pd_u^{n-1} \bigl( u^{n-2} G_{r+1}(u) \bigr), \quad \forall\,n \geqslant 1, \\
\label{E:Thetaunbeta}
\tag{\ref{C:Thetau}b}
  \Theta^r(u^n) &= \sum_{j=0}^r \beta_{r,j} u^{j+1} \pd_u^j (u^{n+j-1})
  = \sum_{j=0}^r \binom{n+j-1}{j} \beta_{r,j} u^{n+j}, \quad \forall\, n \geqslant 0,
\end{align}
where $\beta_{r,j}$ are the coefficients of $G_{r+1}(t)$ in~\eqref{E:gk}.
\end{cor}

\begin{proof}
The proof of \eqref{E:Thetaun} is straightforward, but it is worth noting how the different normalizations of $u(z)$ and $t_{\Lambda_C}(z)$ work out.  From Remark~\ref{R:uvst}, we see that
\begin{multline*}
  \Theta^r (u^n) = \biggl(\frac{-1}{\tpi} \biggr)^r \pd_z^r\bigl( t_{\Lambda_C}(\tpi z)^n \bigr)
  = \biggl(\frac{-1}{\tpi} \biggr)^r \cdot \tpi^r\, \pd_z^r \bigl(t_{\Lambda_C}^n) \big|_{z=\tpi z} \\
  = t^n\, \pd_t^{n-1}\bigl( t^{n-2} G_{r+1}(t) \bigr)\big|_{t=t_{\Lambda_C}(\tpi z)}
  = u^n\, \pd_u^{n-1} \bigl( u^{n-2} G_{r+1}(u) \bigr),
\end{multline*}
where the third equality is~\eqref{E:tndiff}.  The proof of~\eqref{E:Thetaunbeta} is then the same as for~\eqref{E:tndiffbeta}.
\end{proof}

\begin{rem}
We see from~\eqref{E:Thetaun} that there is a duality of some fashion between the $r$-th derivative of $u^n$ and the $(n-1)$-st derivative of $G_{r+1}(u)$, which dovetails with~\eqref{E:diffrSndual}.
\end{rem}

We see from this corollary that $\Theta^r$ can be seen as the operator on power series in $\power{\C}{u}$ given by the following result.  Moreover, from~\eqref{E:Thetafbeta}, we see that computation of $\Theta^r(f)$ is reasonably straightforward once the computation of the coefficients of $G_{r+1}(t)$ can be made.

\begin{cor} \label{C:Thetaf}
Let $f = \sum c_n u^n \in \cU(\Omega)$.  For $r \geqslant 0$,
\begin{align}
\label{E:Thetaf}
\tag{\ref{C:Thetaf}a}
  \Theta^r(f) &= \Theta^r(c_0) +  \sum_{n=1}^\infty c_n u^n\, \pd_u^{n-1} \bigl( u^{n-2} G_{r+1}(u) \bigr), \\
\label{E:Thetafbeta}
\tag{\ref{C:Thetaf}b}
  \Theta^r(f) &= \sum_{j=0}^r \beta_{r,j}\, u^{j+1}\, \pd_u^{j}\bigl( u^{j-1} f \bigr),
\end{align}
where $\beta_{r,j}$ are the coefficients of $G_{r+1}(t)$ in~\eqref{E:gk}.
\end{cor}

Finally we recall the definition of the $r$-th Serre operator $\cD^r$ on modular forms in $M_k^m$ for $r \geqslant 0$.  We set
\begin{equation} \label{E:SerreOp}
  \cD^r(f) := \Theta^r(f)  + \sum_{i=1}^r (-1)^i \binom{k + r-1}{i} \Theta^{r-i}(f)\Theta^{i-1}(E).
\end{equation}
The following result shows that $\cD^r$ takes modular forms to modular forms.

\begin{thm}[{Bosser-Pellarin \cite[Thm.~4.1]{BosserPellarin09}}] \label{T:SerreOp}
For any weight $k$, type $m$, and $r \geqslant 0$,
\[
  \cD^r \bigl( M_k^m \bigr) \subseteq M_{k+2r}^{m+r}.
\]
\end{thm}

\section{$v$-adic modular forms} \label{S:vadicforms}

In this section we review the theory of $v$-adic modular forms introduced by Goss~\cite{Goss14} and Vincent~\cite{Vincent14}.  In~\cite{Serre73}, Serre defined $p$-adic modular forms as $p$-adic limits of Fourier series of classical modular forms and determined their properties, in particular their behavior under the $\vartheta$-operator.  For a fixed finite place $v$ of $K$, Goss and Vincent recently transferred Serre's definition to the function field setting of $v$-adic modular forms, and Goss produced families of examples based on work of Petrov~\cite{Petrov13} (see Theorem~\ref{T:PetrovGoss}).  In \S\ref{S:Thetavadic}, we show that $v$-adic modular forms are invariant under the operators $\Theta^r$.

For our place $v$ of $K$ we fix $\wp \in A_+$, which is the monic irreducible generator of the ideal $\fp_v$ associated to $v$, and we let $d := \deg(\wp)$.  As before we let $A_v$ and $K_v$ denote completions with respect to $v$.

We will write $K \otimes \power{A_v}{u}$ for $K \otimes_A \power{A_v}{u}$, and we recall that $K \otimes\power{A_v}{u}$ can be identified with elements of $\power{K_v}{u}$ that have bounded denominators.  For $f = \sum_{n=0}^\infty c_n u^n \in K \otimes \power{A_v}{u}$, we set
\begin{equation}
  \ord_v(f) := \inf_n \{\ord_v(c_n)\} = \min_n \{ \ord_v(c_n) \}.
\end{equation}
If $\ord_v(f) \geqslant 0$, i.e.,\ if $f\in \power{A_v}{u}$, then we say $f$ is $v$-integral.  For $f$, $g \in K \otimes \power{A_v}{u}$, we write that
\[
  f \equiv g \pmod{\wp^m},
\]
if $\ord_v(f-g) \geqslant m$.  We also define a topology on $K \otimes\power{A_v}{u}$ in terms of the $v$-adic norm,
\begin{equation}
\vnorm{f} := q^{-\ord_v(f)},
\end{equation}
which is a multiplicative norm by Gauss' lemma.

Following Goss, we define the $v$-adic weight space $\bbS = \bbS_v$ by
\begin{equation}
\bbS := \varprojlim_\ell \ZZ/(q^d-1)p^\ell \ZZ = \ZZ/(q^d-1)\ZZ \times \ZZ_p.
\end{equation}
We have a canonical embedding of $\ZZ \hookrightarrow \bbS$, by identifying $n \in \ZZ$ with $(\on,n)$, where $\on$ is the class of $n$ modulo $q^d-1$.  For any $a\in A_+$ with $\wp \nmid a$, we can decompose $a$ as $a = a_1 a_2$, where $a_1 \in A_v^{\times}$ is the $(q^d-1)$-st root of unity satisfying $a_1 \equiv a \pmod{v}$ and $a_2 \in A_v^{\times}$ satisfies $a_2 \equiv 1 \pmod{v}$. Then for any $s = (x,y) \in \bbS$, we define
\begin{equation} \label{E:atos}
a^s := a_1^x a_2^y.
\end{equation}
This definition of $a^s$ is compatible with the usual definition when $s$ is an integer. Furthermore, it is easy to check that the function $s \mapsto a^s$ is continuous from $\bbS$ to $A_v^{\times}$.

\begin{defn}[{Goss~\cite[Def.~5]{Goss14}}] \label{Def:vadic}
We say a power series $f \in K \otimes \power{A_v}{u}$ is a $v$-adic modular form of weight $s \in \bbS$, in the sense of Serre, if there exists a sequence of $K$-rational modular forms $f_i \in M_{k_i}^{m_i}(K)$ so that as $i \to \infty$,
\begin{enumerate}
\item[(a)] $\vnorm{f_i - f} \to 0$,
\item[(b)] $k_i \to s$ in $\bbS$.
\end{enumerate}
Moreover, if $f \neq 0$, then $m_i$ is eventually a constant $m \in \ZZ/(q-1)\ZZ$, and we say that $m$ is the type of~$f$.  We say that $f_i$ converges to $f$ as $v$-adic modular forms.
\end{defn}

It is easy to see that the sum and difference of two $v$-adic modular forms, both with weight $s$ and type $m$, are also $v$-adic modular forms with the same weight and type.  We set
\begin{equation}
  \cM_{s}^m = \bigl\{ f \in K \otimes \power{A_v}{u} \bigm| \textup{$f$ a $v$-adic modular form of weight $s$ and type $m$} \bigr\},
\end{equation}
which is a $K_v$-vector space, and we note that
\[
  \cM_{s_1}^{m_1} \cdot \cM_{s_2}^{m_2} \subseteq \cM_{s_1+s_2}^{m_1+m_2}.
\]
We take $\cM_{s}^m(A_v) := \cM_s^m \cap \power{A_v}{u}$, which is an $A_v$-module.
Moreover, any Drinfeld modular form in $M_k^m(K)$ is also a $v$-adic modular form as the limit of the constant sequence ($u$-expansion coefficients of forms in $M_k^m(K)$ have bounded denominators by~\cite[Thm.~5.13,~\S 12]{Gekeler88}, \cite[Thm.~2.23]{Goss80b}), and so for $k \in \ZZ$, $k \geqslant 0$,
\[
  M_k^m(K) \subseteq \cM_k^m, \quad M_k^m(A) \subseteq \cM_k^m(A_v).
\]
The justification of the final part of Definition~\ref{Def:vadic} is the following lemma.

\begin{lem}
Suppose that $f_i \in M_{k_i}^{m_i}(K)$ converge to a non-zero $v$-adic modular form~$f$.  Then there is some $m \in \ZZ/(q-1)\ZZ$ so that except for finitely terms $m_i = m$.
\end{lem}

\begin{proof}
Since $\vnorm{f-f_i} \to 0$, it follows that $\vnorm{f_i - f_j} \to 0$ as $i$, $j \to \infty$.  If $f = \sum c_n u^n$ and $c_n \neq 0$, then from~\eqref{E:exponents} we see that for $i$, $j \gg 0$, $n \equiv m_i \equiv m_j \pmod{q-1}$.
\end{proof}

\begin{prop}\label{P:limit}
Suppose $\{f_i\}$ is a sequence of $v$-adic modular forms with weights $s_i$.  Suppose that we have $f_0 \in K \otimes \power{A_v}{u}$ and $s_0 \in \bbS$ satisfying,
\begin{enumerate}
\item[(a)] $\vnorm{f_i - f_0} \to 0$,
\item[(b)] $s_i \to s_0$ in $\bbS$.
\end{enumerate}
Then $f_0$ is a $v$-adic modular form of weight $s_0$.  The type of $f_0$ is the eventual constant type of the sequence $\{ f_i \}$.
\end{prop}

\begin{proof}
For each $i \geqslant 1$, we have a sequence of Drinfeld modular forms $g_{i,j} \to f_i$ as $j \to \infty$. Standard arguments show that the sequence of Drinfeld modular forms $\{g_{i,i}\}_{i=1}^\infty$ converges to $f_0$ with respect to the $\vnorm{\,\cdot\,}$-norm and that the weights $k_i$ of $g_{i,i}$ go to $s_0$ in $\bbS$.
\end{proof}

We recall the definitions of Hecke operators on Drinfeld modular forms and their actions on $u$-expansions~\cite[\S 7]{Gekeler88}, \cite[\S 7]{Goss80c}.  For $\ell \in A_+$ irreducible of degree $e$, the Hecke operator $T_\ell : M_k^m \to M_k^m$ is defined by
\[
  (T_\ell f)(z) = \ell^k f(\ell z) + U_\ell f(z) = \ell^k f(\ell z) + \sum_{\beta \in A_{(< e)}} f\biggl( \frac{z+\beta}{\ell} \biggr).
\]
Just as in the classical case the operators $T_\ell$ and $U_\ell$ are uniquely determined by their actions on $u$-expansions.  We define $U_\ell$, $V_\ell : \power{\C}{u} \to \power{\C}{u}$ by
\begin{equation} \label{E:Uell}
  U_\ell \biggl( \sum_{n=0}^\infty c_n u^n \biggr) := \sum_{n=1}^\infty c_n G_{n,\Lambda_\ell} (\ell u),
\end{equation}
where $\Lambda_\ell \subseteq \C$ is the $e$-dimensional $\FF_q$-vector space of $\ell$-division points on the Carlitz module~$C$, and
\begin{equation} \label{E:Vell}
  V_\ell \biggl( \sum_{n=0}^\infty c_n u^n \biggr) := \sum_{n=0}^\infty c_n u_\ell^n.
\end{equation}
We find~\cite[Eq.~(7.3)]{Gekeler88} that $T_\ell : \power{\C}{u} \to \power{\C}{u}$ of weight $k$ is given by $T_\ell = \ell^k V_\ell + U_\ell$.

If $f \in \cM_s^m$ for some weight $s \in \bbS$, then we define $U_\ell(f)$, $V_\ell(f) \in K \otimes \power{A_v}{u}$ as above, and if $\ell \neq \wp$, we set
\begin{equation} \label{E:Tell}
  T_\ell(f) = \ell^s V_{\ell}(f) + U_{\ell} (f),
\end{equation}
where $\ell^s$ is defined as in~\eqref{E:atos} (note that if $\ell = \wp$, then~\eqref{E:atos} is not well-defined).  Of importance to us is that Hecke operators preserve spaces of $v$-adic modular forms.

\begin{prop} \label{P:Heckevadic}
Let $\ell \in A_+$ be irreducible, $\ell \neq \wp$.  For all $v$-adic weights $s$ and types $m$, the operators $T_\ell$, $U_{\wp}$, and $V_{\wp}$ preserve the spaces $\cM_s^m$ and $\cM_s^m(A_v)$.
\end{prop}

We first define a sequence of normalized Eisenstein series studied by Gekeler~\cite[\S 6]{Gekeler88}.  For $d \geqslant 1$, we let
\begin{equation}
  g_d(z) = -L_d \cdot E_{q^d-1}(z),
\end{equation}
which is a Drinfeld modular form of weight $q^d-1$ and type $0$.  By the following proposition we see that $g_d$ plays the role of $E_{p-1}$ for classical modular forms.

\begin{prop}[{Gekeler~\cite[Prop.~6.9, Cor.~6.12]{Gekeler88}}]
\label{P:ge}
For $d \geqslant 1$, the following hold:
\begin{enumerate}
\item[(a)] $g_d \in \power{A}{u};$
\item[(b)] $g_d \equiv 1 \pmod{[d]}$.
\end{enumerate}
\end{prop}

\begin{proof}[Proof of Proposition~\ref{P:Heckevadic}]
Let $f \in \cM_s^m$.   Once we establish that $T_{\ell}(f)$, $U_{\wp}(f)$, and $V_{\wp}(f)$ are elements of $\cM_s^m$, we claim the statement about the operators preserving $\cM_s^m(A_v)$ is a consequence of \eqref{E:Uell}--\eqref{E:Tell}.  Indeed, in either case $\ell \neq \wp$ or $\ell = \wp$ we have $V_{\ell}(\power{A_v}{u}) \subseteq \power{A_v}{u}$, since in \eqref{E:Vell} the $u_\ell^n$ terms are in $\power{A}{u}$ by \eqref{E:uasum}.  Likewise for $U_{\ell}$, the polynomials $G_{n,\Lambda_{\ell}}(\ell u)$ in \eqref{E:Uell} are in $A[u]$, as the $\FF_q$-lattice $\Lambda_\ell$ has exponential function given by polynomials from the Carlitz action, namely $e_{\Lambda_{\ell}}(z) = C_{\ell}(z)/\ell$, and thus by Theorem~\ref{T:Gosspolys}(b),
\[
  \cG_{\Lambda_{\ell}} ( \ell u,x) = \sum_{n=1}^\infty G_{n,\Lambda_{\ell}}(\ell u) x^n
  = \frac{\ell u x}{1 - u C_{\ell}(x)} \in \ell \cdot \power{A[u]}{x}.
\]
Additionally we recall that the cases of $U_{\wp}$ and $V_{\wp}$ preserving $v$-integrality were previously proved by Vincent~\cite[Cor.~3.2, Prop.~3.3]{Vincent14}.

Now by hypothesis we can choose a sequence $\{ f_i \}$ of Drinfeld modular forms of weight $k_i$ and type $m$ so that $f_i \to f$ and $k_i \to s$.  By Proposition~\ref{P:ge}(b), for any $i \geqslant 0$,
\[
  g_d^{q^i} \equiv 1 \pmod{\wp^{q^i}},
\]
since $\ord_v([d]) = 1$.  The form $g_d^{q^i}$ has weight $(q^d-1)q^i$ and type~$0$, and certainly $f_i g_d^{q^i} \to f$ with respect to the $\vnorm{\,\cdot\,}$-norm.  However, we also have that as real numbers,
\[
  \textup{weight of $f_ig_d^{q^i}$} = k_i + (q^d-1)q^i \to \infty, \quad \textup{as $i \to \infty$}.
\]
Therefore, it suffices to assume that $k_i \to \infty$ as real numbers, as $i \to \infty$.

Suppose that $f = \sum c_n u^n$,  $f_i = \sum c_{n,i} u^n \in K \otimes \power{A_v}{u}$.  For $\ell \neq \wp$, since $\ell^{k_i} \to \ell^s$ and $c_{n,i} \to c_n$, we have
\[
  T_{\ell}(f_i) = \ell^{k_i} \sum_{n=0}^\infty c_{n,i} u_\ell^n +
    \sum_{n=0}^\infty c_{n,i} G_{n,\Lambda_\ell}(\ell u)
  \; \longrightarrow \;  T_{\ell}(f).
\]
Since $T_{\ell}(f_i) \in M_{k_i}^m(K)$, it follows that $T_{\ell}(f) \in \cM_s^m$.

Now consider the case $\ell = \wp$.  Since $k_i \to \infty$, we see that $|\wp^{k_i}|_v \to 0$. Therefore,
\[
  T_{\wp}(f_i) \to \sum_{n=0}^\infty c_n G_{n,\Lambda_{\wp}}(\wp u) = U_{\wp}(f),
\]
and so $U_{\wp}(f) \in \cM_s^m$.  By the same argument each $U_{\wp}(f_i) \in \cM_{k_i}^m$, starting with the constant sequence $f_i$ in the first paragraph.  By subtraction each
\begin{equation} \label{E:Vfi}
  V_{\wp}(f_i) =  \wp^{-k_i} \bigl( T_{\wp}(f_i) - U_{\wp}(f_i) \bigr)
\end{equation}
is then an element of $\cM_{k_i}^m$.  Because $c_{n,i} \to c_n$, we see from~\eqref{E:Vell} that $V_{\wp}(f_i) \to V_{\wp}(f)$ with respect to the $\vnorm{\,\cdot\,}$-norm.  Thus by Proposition~\ref{P:limit}, $V_{\wp}(f) \in \cM_s^m$ as desired.
\end{proof}

\section{Theta operators on $v$-adic modular forms} \label{S:Thetavadic}

As is well known the operators $\Theta^r$ do not generally take Drinfeld modular forms to Drinfeld modular forms~\cite{BosserPellarin08}, \cite{UchinoSatoh98}.  However, we will prove in this section that each $\Theta^r$, $r\geqslant 0$, does preserve spaces of $v$-adic modular forms.  Using the equivalent formulations in~\eqref{E:Thetaf} and~\eqref{E:Thetafbeta}, we define $K_v$-linear operators
\[
  \Theta^r : K \otimes \power{A_v}{u} \to K \otimes \power{A_v}{u}, \quad r \geqslant 0.
\]

\begin{thm} \label{T:Thetavadic}
For any weight $s \in \bbS$ and type $m \in \ZZ/(q-1)\ZZ$, we have for $r \geqslant 0$,
\[
  \Theta^r : \cM_{s}^m \to \cM_{s+2r}^{m+r}.
\]
\end{thm}

This can be seen as similar in spirit to the results of Bosser and Pellarin~\cite[Thm.~2]{BosserPellarin08}, \cite[Thm.~4.1]{BosserPellarin09} (see also Theorem~\ref{T:SerreOp}), that $\Theta^r$ preserves spaces of Drinfeld quasi-modular forms, and our main arguments rely on essentially showing that quasi-modular forms with $K_v$-coefficients are $v$-adic and applying Theorem~\ref{T:SerreOp}. Consider first the operator $\Theta = \Theta^1$, which can be equated by~\eqref{E:uder1z} with the operation on $u$-expansions given by
\[
  \Theta = u^2\, \pd_u^1.
\]
We recall a formula of Gekeler~\cite[\S 8]{Gekeler88} (take $r=1$ in~\eqref{E:SerreOp}), which states that for $f \in M_k^m$,
\[
  \Theta(f) = \cD^1(f) + kEf,
\]
where $E$ is the false Eisenstein series whose $u$-expansion is given in~\eqref{E:falseexp}.  Our first goal is to show that $E$ is a $v$-adic modular form, for which we use results of Goss and Petrov.  For $k$, $n \geqslant 1$ and $s \in \bbS$, we set
\begin{equation}
  f_{k,n} := \sum_{a \in A_+} a^{k-n} G_n(u_a), \quad  \fhat_{s,n} := \sum_{\substack{a \in A_+\\ \wp\, \nmid\, a}} a^s G_{n}(u_a).
\end{equation}
The notation $f_{k,n}$ and $\fhat_{s,n}$ is not completely consistent, since $f_{k,n}$ is more closely related to $\fhat_{k-n,n}$ than $\fhat_{k,n}$, but this viewpoint is convenient in many contexts (see~\cite{Goss14}).

\begin{thm}[{Goss~\cite[Thm.~2]{Goss14}, Petrov~\cite[Thm.~1.3]{Petrov13}}] \label{T:PetrovGoss} \
\begin{enumerate}
\item[(a)] \textup{(Petrov)} Let $k$, $n \geqslant 1$ be chosen so that $k-2n > 0$, $k \equiv 2n \pmod{q-1}$, and $n \leqslant p^{\ord_p(k-n)}$.  Then
\[
f_{k,n} \in M_k^m(K),
\]
where $m \equiv n \pmod{q-1}$.
\item[(b)] \textup{(Goss)} Let $n \geqslant 1$.  For $s = (x,y) \in \bbS$ with $x \equiv n \pmod{q-1}$ and $y \equiv 0 \pmod{q^{\lceil \log_q(n) \rceil}}$, we have
\[
\fhat_{s,n} \in \cM_{s+n}^m,
\]
where $m \equiv n \pmod{q-1}$.
\end{enumerate}
\end{thm}

We note that the statement of Theorem~\ref{T:PetrovGoss}(b) is slightly stronger than what is stated in~\cite{Goss14}, but Goss' proof works here without changes.  We then have the following corollary.

\begin{cor}
For any $\ell \equiv 0 \pmod{q-1}$, we have $\fhat_{\ell+1,1} \in \cM_{\ell+2}^1$.
\end{cor}

If we take $\ell = 0$, we see that
\[
  \fhat_{1,1} = \sum_{\substack{a \in A_+\\ \wp\, \nmid\, a}} a u_a \in \power{A_v}{u}
\]
is a $v$-adic modular form in $\cM_2^1(A_v)$ and is a partial sum of $E$ in~\eqref{E:falseexp}.  From this we can prove that $E$ itself is a $v$-adic modular form.

\begin{thm} \label{T:Evadic}
The false Eisenstein series $E$ is a $v$-adic modular form in $\cM_{2}^1(A_v)$.
\end{thm}

\begin{proof}
Starting with the expansion in \eqref{E:falseexp}, we see that $E \in \power{A_v}{u}$.  Also,
\begin{align*}
  E = \sum_{a \in A_+} a u_a &= \sum_{\substack{a \in A_+\\ \wp\, \nmid\, a}} a u_a + \wp \sum_{a \in A_+} a u_{\wp a} \\
  &= \sum_{\substack{a \in A_+\\ \wp\, \nmid\, a}} a u_a + \wp \sum_{\substack{a \in A_+\\ \wp\, \nmid\, a}} a u_{\wp a}
  + \wp^2 \sum_{a \in A_+} a u_{\wp^2 a},
\end{align*}
and continuing in this way, we find
\[
  E = \sum_{j=0}^\infty \Biggl( \wp^j \sum_{\substack{a \in A_+\\ \wp\, \nmid\, a}} a u_{\wp^j a} \Biggr).
\]
We note that
\[
  \sum_{\substack{a \in A_+\\ \wp\, \nmid\, a}} a u_{\wp^j a} = V_{\wp}^{\circ j} \Biggl( \sum_{\substack{a \in A_+\\ \wp\, \nmid\, a}} a u_a \Biggr) = V_{\wp}^{\circ j} (\fhat_{1,1}),
\]
where $V_{\wp}^{\circ j}$ is the $j$-th iterate $V_{\wp} \circ \cdots \circ V_{\wp}$.  By Proposition~\ref{P:Heckevadic}, we see that $V_{\wp}^{\circ j}(\fhat_{1,1}) \in \cM_2^1(A_v)$ for all $j$.  Moreover,
\[
  E = \sum_{j=0}^\infty \wp^j V_{\wp}^{\circ j}(\fhat_{1,1}),
\]
the right-hand side of which converges with respect to the $\vnorm{\,\cdot\,}$-norm, and so we are done by Proposition~\ref{P:limit}.
\end{proof}

\begin{proof}[Proof of Theorem~\ref{T:Thetavadic}]
Let $f \in \cM_s^m$ and pick $f_i \in M_{k_i}^m(K)$ with $f_i \to f$.  It follows from the formulas in Corollary~\ref{C:Thetaf} that $\Theta^r(f_i) \to \Theta^r(f)$ with respect to the $\vnorm{\,\cdot\,}$-norm for each $r \geqslant 0$, so by Proposition~\ref{P:limit} it remains to show that each
\[
  \Theta^r(f_i) \in \cM_{k_i+2r}^{m+r}.
\]
We proceed by induction on $r$.  If $r = 1$, then since $\cD^1(f_i) \in M_{k_i+2}^{m+1}(K)$ for each $i$ by Theorem~\ref{T:SerreOp}, it follows from Theorem~\ref{T:Evadic} that
\[
  \Theta(f_i) = \cD^1(f_i) + k_i Ef_i \in \cM_{k_i+2}^{m+1},
\]
for each $i$.  Now by \eqref{E:SerreOp}, for each $i$
\[
  \Theta^r(f_i) = \cD^r(f_i) - \sum_{j=1}^r (-1)^j \binom{k_i + r-1}{j} \Theta^{r-j}(f_i) \Theta^{j-1}(E).
\]
By Theorem~\ref{T:SerreOp}, $\cD^r(f_i) \in M_{k_i+2r}^{m+r}(K)$, and by the induction hypothesis and Theorem~\ref{T:Evadic} the terms in the sum are in $\cM_{k_i+2r}^{m+r}$.
\end{proof}

\section{Theta operators and $v$-adic integrality} \label{S:vintegral}

We see from Theorem~\ref{T:Thetavadic} that $\Theta^r : \cM_s^m \to \cM_{s+2r}^{m+r}$, and it is a natural question to ask whether $\Theta^r$ preserves $v$-integrality, i.e.,
\[
  \Theta^r : \cM_s^m(A_v) \stackrel{?}{\to} \cM_{s+2r}^{m+r}(A_v).
\]
However, it is known that this can fail for $r$ sufficiently large because of the denominators in $G_{r+1}(u)$ (e.g., see Vincent~\cite[Cor.~1]{Vincent15}).  Nevertheless, in this section we see that $\Theta^r$ is not far off from preserving $v$-integrality.

For an $A$-algebra $R$ and a sequence $\{ b_m \} \subseteq R$, we define an $A$-Hurwitz series over $R$ (cf.\ \cite[\S 9.1]{Goss}) by
\begin{equation}
  h(x) = \sum_{m=0}^\infty \frac{b_m}{\Pi_m} x^m \in \power{(K \otimes_A R)}{x},
\end{equation}
where we recall the definition of the Carlitz factorial $\Pi_m$ from~\eqref{E:Pi}.  Series of this type were initially studied by Carlitz~\cite[\S 3]{Carlitz37} and further investigated by Goss~\cite[\S 3]{Goss78}, \cite[\S 9.1]{Goss}.  The particular cases we are interested in are when $R = A$ or $R=A[u]$, but we have the following general proposition whose proof can be easily adapted from \cite[\S 3.2]{Goss78}, \cite[Prop.~9.1.5]{Goss}.

\begin{prop} \label{P:Hurwitz}
Let $R$ be an $A$-algebra, and let $h(x)$ be an $A$-Hurwitz series over $R$.
\begin{enumerate}
\item[(a)] If the constant term of $h(x)$ is $1$, then $1/h(x)$ is also an $A$-Hurwitz series over $R$.
\item[(b)] If $g(x)$ is an $A$-Hurwitz series over $R$ with constant term $0$, then $h(g(x))$ is also an $A$-Hurwitz series over $R$.
\end{enumerate}
\end{prop}

We apply this proposition to the generating function of Goss polynomials.

\begin{lem} \label{L:PiGoss}
For each $k \geqslant 1$, we have $\Pi_{k-1} G_{k}(u) \in A[u]$.
\end{lem}

\begin{proof}
Consider the generating series
\[
  \frac{\cG(u,x)}{x} = \sum_{k=1}^\infty G_{k}(u) x^{k-1} = \frac{u}{1- ue_C(x)}.
\]
We claim that $\cG(u,x)/x$ is an $A$-Hurwitz series over $A[u]$.  Indeed certainly the constant series $u$ itself is one, and
\[
  1 - ue_C(x) = 1 - \sum_{i=0}^\infty \frac{u x^{q^i}}{\Pi_{q^i}}
\]
is an $A$-Hurwitz series over $A[u]$ with constant term~$1$, so the claim follows from Proposition~\ref{P:Hurwitz}(a).  The result is then immediate.
\end{proof}

\begin{thm} \label{T:vintegral}
For $r \geqslant 0$, if $f \in \cM_s^m(A_v)$, then $\Pi_{r} \Theta^r(f) \in \cM_{s+2r}^{m+r}(A_v)$.  Thus we have a well-defined operator,
\[
  \Pi_{r}\Theta^r : \cM_s^m(A_v) \to \cM_{s+2r}^{m+r}(A_v).
\]
\end{thm}

\begin{proof}
By~\eqref{E:Thetaf}, we see that the possible denominators of $\Theta^r(f)$ come from the denominators of $G_{r+1}(u)$, which are cleared by $\Pi_{r}$ using Lemma~\ref{L:PiGoss}.
\end{proof}

\begin{rem}
Once we see that $\Pi_{r}\Theta^r$ preserves $v$-integrality, the question of whether $\Pi_{r}$ is the best possible denominator is important but subtle, and in general the answer is no.  For example, taking $r = q^{d+1}-1$, we see from Theorem~\ref{T:Gosspolys}(d), that
\[
  G_{q^{d+1}}(u) = u^{q^{d+1}},
\]
and so $\Theta^{q^{d+1}-1} : \power{A_v}{u} \to \power{A_v}{u}$ already by \eqref{E:Thetaf}.  However, $\Pi_{q^{d+1}-1}$ can be seen to be divisible by $\wp$.

Nevertheless, we do see that $\Pi_{r}$ is the best possible denominator in many cases.  For example,
let $r = q^i$ for $i \geqslant 1$.  Then from Theorem~\ref{T:Gosspolys}(a),
\[
  G_{q^i+1}(u) = u\biggl( G_{q^i}(u) + \frac{G_{q^i+1-q}(u)}{D_1} + \cdots +
  \frac{G_{1}(u)}{D_i} \biggr) = u\biggl( u^{q^i} + \frac{G_{q^i+1-q}(u)}{D_1} + \cdots +
  \frac{u}{D_i} \biggr).
\]
{From} Theorem~\ref{T:Gosspolys}(b) we know that $u^2$ divides $G_k(u)$ for all $k \geqslant 2$, and so we find that the coefficient of $u^2$ in $G_{q^i+1}(u)$ is precisely $1/D_i$, which is the same as $\Pi_{q^i}$.
\end{rem}

\end{document}